\documentclass[12pt,draft]{article}
\usepackage[russian]{babel}
\usepackage{amsmath,amsthm,amssymb}
\newtheorem{stat}{Proposition}[section]

\newtheorem{ex}[stat]{Example}
\newtheorem{thr}[stat]{Theorem}

\newtheorem{hyp}[stat]{Conjecture}
\newtheorem{quest}[stat]{Question}
\numberwithin{equation}{section}
\theoremstyle{definition}
\newtheorem{defn}[stat]{Definition}

\def\R{\mathbb{R}}

\def\M{{\cal M}}

\def\S{{\cal S}}

\def\trop{rk_t}
\def\Kap{rk_K}
\def\Kf{\textbf{K}}
\def\Kmat{\mathcal K}

\begin{document}
\renewcommand{\abstractname}{}
\renewcommand{\refname}{Bibliography}
\renewcommand{\proofname}{Proof}



\begin{center}
\bf{\textsc{EXAMPLE OF A 6-BY-6 MATRIX WITH DIFFERENT TROPICAL AND KAPRANOV RANKS}}
\end{center}


\begin{center}
\textsc{yaroslav shitov}
\end{center}

\medskip

\parshape=10
1.3cm 11.1cm 1.3cm 11.1cm 1.3cm 11.1cm 1.3cm 11.1cm 1.3cm 11.1cm 1.3cm 11.1cm 1.3cm 11.1cm 1.3cm 11.1cm 1.3cm 11.1cm 1.3cm 11.1cm
\textsc{Abstract.} We provide an example of a $6$-by-$6$ matrix $A$ such that $rk_t(A)=4$, $rk_K(A)=5$.
This answers a question asked by M.~Chan, A.~Jensen, and E.~Rubei.

\smallskip

\parshape=10
1.3cm 11.1cm 1.3cm 11.1cm 1.3cm 11.1cm 1.3cm 11.1cm 1.3cm 11.1cm 1.3cm 11.1cm 1.3cm 11.1cm 1.3cm 11.1cm 1.3cm 11.1cm 1.3cm 11.1cm
\textsc{Keywords}: matrix theory, tropical semiring, tropical rank, Kapranov rank.

\smallskip

\parshape=10
1.3cm 11.1cm 1.3cm 11.1cm 1.3cm 11.1cm 1.3cm 11.1cm 1.3cm 11.1cm 1.3cm 11.1cm 1.3cm 11.1cm 1.3cm 11.1cm 1.3cm 11.1cm 1.3cm 11.1cm
\textit{Mathematics Subject Classification:} 15A03, 15A15.

\bigskip\medskip


\section{Introduction}

We work over the \textit{tropical semiring} $(\mathbb{R},\oplus,\otimes)$ whose operations are
$$a\oplus b=\min\{a,b\},\mbox{ }a\otimes b=a+b.$$
We consider \textit{tropical matrices}, i.e. matrices over the tropical semiring.
There exist many different ways to define the rank of a tropical matrix, see~\cite{AGG,DSS}.
We deal with the notions of tropical rank and Kapranov rank, see also~\cite{CJR, KR}.

\begin{defn}\label{perm}
We define the \textit{permanent} of a tropical matrix $S\in\mathbb{R}^{n\times n}$ as
\begin{equation}\label{def1}perm(S)=\min\limits_{\sigma\in\S_n} \{s_{1,\sigma(1)}+\ldots+s_{n,\sigma(n)}\}.\end{equation}
\end{defn}

\begin{defn}\label{deftropdeg}
The matrix $S$ is called \textit{tropically singular} if the minimum in~(\ref{def1}) is attained at least twice.
Otherwise, $S$ is called \textit{tropically non-singular}.
\end{defn}

\begin{defn}\label{deftrop}
The \textit{tropical rank} of a matrix $M\in\mathbb{R}^{p\times q}$ is the largest integer $r$ such that
$M$ has a tropically non-singular $r$-by-$r$ submatrix. We denote the tropical rank of $M$ by $\trop(M)$.
\end{defn}

Let $\Kf$ denote the field whose elements are formal sums
$$a(t)=\sum_{i=1}^{\infty}a_it^{\alpha_i} \mbox{ such that } a_n\in\mathbb{C}, \alpha_n\in\mathbb{R}, \lim_{n\rightarrow\infty}\alpha_n=+\infty.$$
Let $deg: \Kf^*\rightarrow\mathbb{R}$ be a natural valuation sending $a(t)$ to the least of the exponents $\alpha_i$, i.e.
$deg(a)=\min_{n:a_n\neq0}\{\alpha_n\}.$ By definition, assume $deg(0)=\infty$.
We say that $B\in\Kf^{m\times n}$ is a \textit{lift} of $T\in\mathbb{R}^{m\times n}$ if $deg(b_{ij})=t_{ij}$ for any $i,j$.
The notion of the Kapranov rank of a matrix can be defined in the following way, see~\cite[Corollary 3.4]{DSS}.

\begin{defn}\label{defKap}
Let $M\in\mathbb{R}^{m\times n}$. We define the Kapranov rank of $M$ as
$$\Kap(M)=\min_{\Kmat_M}\{rank(\Kmat_M)\},$$
where the minimum is taken over all lifts of $M$. The expression $rank(\Kmat_M)$ means the
usual rank of a matrix $\Kmat_M$ over the field $\Kf$.
\end{defn}

The notion of Kapranov rank was deeply investigated in~\cite{CJR, DSS, KR}.
Develin, Santos, and Sturmfels in~\cite{DSS} show that $\Kap(M)\geqslant\trop(M)$ for every matrix $M$.
The following theorem points out the connection with matroids.

\begin{thr}\cite[Corollary 7.4]{DSS}\label{DSSthr}
Let $\M$ be a matroid which is not representable over $\mathbb{C}$.
Then the Kapranov and tropical ranks of the cocircuit matrix ${\cal C}(\M)$ are different.
\end{thr}

Theorem~\ref{DSSthr} makes it possible to construct examples of matrices with different tropical and Kapranov ranks.
The example of a $7$-by-$7$ matrix with different ranks is provided in~\cite{DSS}.

Kim and Roush in~\cite{KR} mostly deal with algorithmical aspects of the Kapranov rank. They prove that
determining Kapranov rank of tropical matrices is NP-hard.
Also, in~\cite{KR} it was shown that there exist matrices of tropical rank
$3$ and arbitrarily high Kapranov rank.

The following theorem was proven in~\cite{CJR}.

\begin{thr}\cite[Corollary 1.5]{CJR}\label{CJRthr}
Let $M\in\R^{m\times n}$, $\min\{m,n\}\leqslant5$. Then $\Kap(M)=\trop(M)$.
\end{thr}

Chan, Jensen, and Rubei in~\cite{CJR} point out the connection with the notion of tropical basis.
They ask the following question.

\begin{quest}\cite[Question 1.1]{CJR}\label{hy1}
For which numbers $d$, $n$, and $r$ do the $(r+1)\times(r+1)$-minors of a d-by-n matrix form a tropical basis? Equivalently, for which
$d$, $n$, $r$ does every $d$-by-$n$ matrix of tropical rank at most $r$ have Kapranov rank at most $r$?
\end{quest}

In~\cite{CJR} the following conjecture was also made.

\begin{hyp}\cite[Conjecture 1.6]{CJR}\label{hy}
The $(r + 1)\times(r + 1)$ minors of a d-by-n matrix are a tropical basis if and only if either $r\leqslant2$ or $r\geqslant\min\{d,n\}-2$.
\end{hyp}

Also, in~\cite{CJR} it was asked whether there exists a $6$-by-$6$ matrix with different tropical and Kapranov ranks.
We answer this question by providing an example of a $6$-by-$6$ matrix with tropical rank $4$ and Kapranov rank $5$.

Now let us take into account the equivalence given in Question~\ref{hy1}.
Our example shows that the 5-by-5 minors of a 6-by-6 matrix are not a tropical basis.
Thus we disprove Conjecture~\ref{hy}.

Additionally, we note that the difference between the tropical and Kapranov ranks of our matrix does not have a matroidal nature.
Indeed, matroids with at most $6$ elements are all representable over $\mathbb{C}$, see~\cite{BCH}.

\section{The Example}

\begin{ex}\label{6-6}
Let $$A=\left(
\begin{array}{cccccc}
0 & 0 & 4 & 4 & 4 & 4 \\
0 & 0 & 2 & 4 & 1 & 4 \\
4 & 4 & 0 & 0 & 4 & 4 \\
2 & 4 & 0 & 0 & 2 & 4 \\
4 & 4 & 4 & 4 & 0 & 0 \\
2 & 4 & 1 & 4 & 0 & 0 \\
\end{array}
\right)
.$$
Then $\trop(A)=4$, $\Kap(A)=5$.
\end{ex}

\begin{proof}
1. Note that every $5$-by-$5$ submatrix
of $A$ can be written in some of the following forms (up to permutations of rows and columns):
$$S'=\left(
\begin{array}{cccccc}
0 & s'_{12} & s'_{13} & s'_{14} & s'_{15} \\
s'_{21} & 0 & 0 & s'_{24} & s'_{25} \\
s'_{31} & 0 & 0 & s'_{34} & s'_{35} \\
s'_{41} & s'_{42} & s'_{43} & 0 & 0 \\
s'_{51} & s'_{52} & s'_{53} & 0 & 0 \\
\end{array}
\right), \mbox{ }
S''=\left(
\begin{array}{cccccc}
0 & 4 & 4 & 4 & 4 \\
0 & x & 4 & y & 4 \\
s''_{31} & 0 & 0 & 4 & 4 \\
s''_{41} & 0 & 0 & z & 4 \\
s''_{51} & s''_{52} & s''_{53} & 0 & 0 \\
\end{array}
\right),$$
where $x,y,z\in\{1,2\}$, $s'_{ij},s''_{ij}\in\{1,2,4\}$.
By Definition~\ref{perm}, $perm(S')=0$. The minimum in~(\ref{def1}) for $S'$ is given by $id,(23)\in\S_5$.
Analogously, $perm(S'')=y$, the minimum is given by $(24), (243)\in\S_5$.
Thus by Definition~\ref{deftropdeg}, every $5\times 5$-submatrix of $A$ is tropically singular.
From Definition~\ref{deftrop} it follows that $\trop(A)\leqslant4$.

Now consider the $4$-by-$4$ submatrix which is formed by the $1$st, $2$nd, $4$th, and $6$th rows and
the $1$st, $4$th, $5$th, and $6$th columns of $A$:
$$
\left(
\begin{array}{cccccc}
0 & 4 & 4 & 4 \\
0 & 4 & 1 & 4 \\
2 & 0 & 2 & 4 \\
2 & 4 & 0 & 0 \\
\end{array}
\right).$$
The minimum in the expression for its permanent is given by the only permutation $(23)\in\S_4$.
Thus by Definition~\ref{deftrop}, $\trop(A)=4$.

2. Let us consider the matrix
$$M_0=\left(
\begin{array}{cccccc}
1 & 1 & t^4 & t^4 & t^4 & t^4 \\
-1 & -1 & t^2 & t^4 & t & t^4 \\
t^4 & t^4 & 1-t^2 & 1 & -t^4 & -t^4 \\
t^2 & t^4 & -1-t & -1 & t^2 & -t^4 \\
-t^4 & -t^4 & -t^4 & -t^4 & -1-t^2 & 1 \\
-t^2 & -t^4 & t & -t^4 & 1-t & -1 \\
\end{array}
\right)\in\Kf^{6\times6},$$
which is a lift of $A$. The sum of the rows of $M_0$ is the zero row, so that
the rank of $M_0$ is at most $5$. Thus by Definition~\ref{defKap}, $\Kap(A)\leqslant 5$.

Now let $H\in\Kf^{6\times6}$ be an arbitrary lift of $A$.
It follows directly from definitions that $deg(ab)=deg(a)+deg(b)$, $deg(a+b)\geqslant\min\{deg(a),deg(b)\}$ for any $a,b\in\Kf$.
Since $deg(h_{pq})=a_{pq}$ for any $p,q$, we obtain the following expression for the minor $H_{25}$:
$$H_{25}=h_{12}h_{34}h_{41}h_{56}h_{63}+h_{12}h_{33}h_{44}h_{56}h_{61}-h_{12}h_{34}h_{43}h_{56}h_{61}+g_1,$$
where $deg(g_1)\geqslant4$. Analogously, the minor $H_{61}$ can be expressed as
$$H_{61}=h_{12}h_{25}h_{33}h_{44}h_{56}-h_{12}h_{25}h_{34}h_{43}h_{56}+g_2,\mbox{ }deg(g_2)\geqslant4.$$

We denote $\Delta=h_{33}h_{44}-h_{34}h_{43}$, $\delta=deg(\Delta)$. We obtain
$$H_{25}=h_{12}h_{34}h_{41}h_{56}h_{63}+h_{12}\Delta h_{56}h_{61}+g_1,\mbox{ }deg(h_{12}h_{34}h_{41}h_{56}h_{63})=3,$$
\begin{equation}\mbox{$ $ $ $ $ $ $ $ $ $ $ $ $ $ $ $ $ $  $ $ $ $ $ $ $ $ $ $ $ $ $ $ $ $  $ $ $ $ $ $ $ $ $ $ $ $ $ $ $ $  $ $ $ $ $ $ $ $ $ $ $ $ $ $ $ $ $ $ $ $ $ $ $ $ $ $ $ $ $ $ $ $  $ $ $ $ $ $ $ $ $ $ $ $ $ $ $ $  $ $ $ $ $ $ $ $ $ $ $ $ $ $ $ $  $ $ $ $ $ $ $ $ $ $ $ $}\label{eqex1}\mbox{ }deg(h_{12}\Delta h_{56}h_{61})=2+\delta;\end{equation}
\begin{equation}\label{eqex2}H_{61}=h_{12}h_{25}\Delta h_{56}+g_2,
\mbox{ }deg(h_{12}h_{25}\Delta h_{56})=1+\delta.\end{equation}

It follows from definitions that $deg(v_1+v_2)=\min\{deg(v_1),deg(v_2)\}$
for any $v_1,v_2\in\Kf$ such that $deg(v_1)\neq deg(v_2)$.
Thus if $\delta>1$, then from~(\ref{eqex1}) it follows that $deg(H_{25})=3$, i.e. $H_{25}\neq0$.
Analogously, if $\delta<1$, then $deg(H_{25})=2+\delta$, i.e. $H_{25}\neq0$.
Finally, if $\delta=1$, then from~(\ref{eqex2}) it follows that $deg(H_{61})=2$, i.e. $H_{61}\neq0$.
We see that some of the minors $H_{25}$ and $H_{61}$ differs from $0$. This shows that the rank
of $H$ is at least $5$. By Definition~\ref{defKap}, $\Kap(A)\geqslant5$.
The proof is complete.
\end{proof}

\begin{thr}\label{genc}
The matrix $A$ from Example~\ref{6-6} contains the least number of rows and the least number of columns
among tropical matrices $M$ such that $\Kap(M)\neq\trop(M)$.
\end{thr}

\begin{proof}
Follows from Theorem~\ref{CJRthr} and Example~\ref{6-6}.
\end{proof}

\section{Acknowledgements}
I would like to thank my scientific advisor Professor Alexander E. Guterman for constant attention to my work and fruitful discussions.

\textsc{Faculty of Algebra, Department of Mathematics and Me\-chanics, Moscow State University, GSP-1, 119991 Moscow, Rus\-sia.}

\textit{E-mail:} \verb"yaroslav-shitov@yandex.ru"
\end{document}